\documentclass[12pt]{article}
\usepackage{amsmath, amsthm, amssymb}
\usepackage{array,color,colortbl,makecell} 

\theoremstyle{plain}
\newtheorem{theorem}{Theorem}
\newtheorem{lemma}[theorem]{Lemma}
\newtheorem{example}[theorem]{Example}

\newtheorem{corollary}[theorem]{Corollary}

\newcommand{\mcol}[1]{\multicolumn{1}{|c|}{#1}}
\newcommand{\vmcol}[2]{\multicolumn{1}{!{\vrule width#1}c|}{#2}}

\newcommand{\hz}{\hphantom{0}}

\renewcommand{\geq}{\geqslant}
\renewcommand{\leq}{\leqslant}

\def\dfrac#1#2{\lower0.15ex\hbox{\large$\frac{#1}{#2}$}} 

\title{The maximum, supremum and spectrum for critical set sizes in $(0,1)$-matrices}
\author{
Nicholas J. Cavenagh and Liam K. Wright, \\
Department of Mathematics, \\
The University of Waikato, \\
Private Bag 3105, \\
Hamilton 3240, New Zealand \\
\texttt{nickc@waikato.ac.nz} \\
 \texttt{liam.wright479@live.com} \\
}

\begin{document}

\date{}
\maketitle

\begin{abstract}
If $D$ is a partially filled-in $(0,1)$-matrix with
a unique completion to a $(0,1)$-matrix $M$ (with prescribed row and column sums), we say that $D$ is a {\em defining set} for $M$. 
A {\em critical set} is a minimal defining set (the deletion of any entry results in more than one completion).
We give a new classification of critical sets in $(0,1)$-matrices and apply this theory to $\Lambda_{2m}^m$, the set of $(0,1)$-matrices of dimensions $2m\times 2m$ with uniform row and column sum $m$. 

The smallest possible size for a defining set of a matrix in $\Lambda_{2m}^m$ is $m^2$ 
 \cite{Cav}, and the infimum (the largest smallest defining set size for members of $\Lambda_{2m}^m$) is known asymptotically \cite{CR}. 
 We show that no critical set of size larger than $3m^2-2m$ exists in an element of $\Lambda_{2m}^m$ and that there exists a critical set of size $k$ in an element of $\Lambda_{2m}^m$ for each $k$ such that $m^2\leq k\leq 3m^2-4m+2$.   
We also bound the supremum (the smallest largest critical set size for members of $\Lambda_{2m}^m$) between $\lceil (3m^2-2m+1)/2\rceil$ and $2m^2-m$. 
 
\end{abstract}

{\noindent \bf Keywords:} $(0,1)$-matrix, critical set, defining set, frequency square, $F$-square. 

\section{Introduction}
Where convenient, we keep notation consistent with \cite{Bru}.

Let $R=(r_1,r_2,\dots ,r_m)$ and $S=(s_1,s_2,\dots ,s_n)$ be vectors of non-negative integers such that $\sum_{i=1}^m r_i=\sum_{j=1}^n s_j$. 
Then ${\mathcal A}(R,S)$ is defined to be the set of all $m\times n$ $(0,1)$-matrices with 
$r_i$ $1$'s in row $i$ and $s_j$ $1$'s in column $j$, where $1\leq i\leq m$ and $1\leq j\leq n$. 
If $M\in {\mathcal A}(R,S)$, we 
refer to $R$ and $S$ as the row sum and column sum vectors for $M$, respectively.

With $R$ and $S$ as above, we next define $\mathcal A'(R,S)$ to be the set of all $m\times n$ $(0,1,\star)$-matrices with:
\begin{enumerate}
\item at most 
$r_i$ $1$'s in row $i$,
\item  at most $n-r_i$ $0$'s in row $i$,
\item at most $s_j$ $1$'s in column $j$, 
\item at most $m-s_j$ $0$'s in column $j$.
\end{enumerate}
We call a matrix $M\in\mathcal A'(R,S)$ a {\em partial} $(0,1)$-matrix
with row sum vector $R$ and column sum vector $S$. 
If $M_{ij}=\star$ we say that position $(i,j)$ is {\em empty}. 
Indeed, ${\mathcal A}(R,S)\subseteq {\mathcal A}'(R,S)$ and a partial $(0,1)$-matrix is
a $(0,1)$-matrix if and only if it none of its positions are empty (equal to $\star$). 
(Note that our definition of a partial $(0,1)$-matrix allows for the possibility of a matrix 
 $M\in\mathcal A'(R,S)$ which has {\em no} completion to a  
 a $(0,1)$-matrix in ${\mathcal A}(R,S)$.) 

We sometimes consider a partial $(0,1)$-matrix $M$ as a set of triples
$$M=\{(i,j,M_{ij})\mid 1\leq i\leq m, 1\leq j\leq n, M_{ij}\in \{0,1\}\}.$$ 
Note that we naturally omit here the empty positions of $M$. 
It is usually clear by context whether we are considering a 
partial $(0,1)$-matrix as a matrix or as a set of ordered triples.
For example, if we say that $M_1\subseteq M_2$ for two partial $(0,1)$-matrices $M_1$ and $M_2$, we are considering $M_1$ and $M_2$ as sets.
Similarly, the {\em size} of a partial $(0,1)$-matrix $M$ refers to $|M|$ where $M$ is a set (i.e. the number of $0$'s and $1$'s).

With this is mind, 
suppose that $M\in {\mathcal A}(R,S)$ and 
$D\in {\mathcal A}'(R,S)$. 
We say that $D$ is a {\em defining set} for $M$ if
$M$ is 
the unique member of ${\mathcal A}(R,S)$ such that $D\subseteq M$.  
We say that $D$ is a {\em critical set} for $M$ if it is minimal with respect to this property. 
That is, $D$ is a critical set for $M$ if it is a defining set for $M$ and if
$D'\subset D$, then $D'$ is not a defining set for $D$. 

This is analogous to the usual definition of defining sets 
and critical sets 
of Latin squares and other combinatorial designs (\cite{St,Ke}). 

Given a $(0,1)$-matrix $M$, the size of the smallest (respectively, largest) critical set 
in $M$ is denoted by scs$(M)$ (respectively, lcs$(M)$). 
For given row column sum vectors $R$ and $S$, 
scs$({\mathcal A}(R,S))$ 
(respectively, lcs$({\mathcal A}(R,S))$) 
is the size of the smallest 
(respectively, largest) 
critical set for all members of the set ${\mathcal A}(R,S)$. 
More precisely,
$$\mbox{\rm scs}({\mathcal A}(R,S))=\mbox{\rm min}\{\mbox{\rm scs}(M) 
\mid M\in {\mathcal A}(R,S)\},$$
$$\mbox{\rm lcs}({\mathcal A}(R,S))=\mbox{\rm max}\{\mbox{\rm lcs}(M) 
\mid M\in {\mathcal A}(R,S)\}.$$
We also define a type of infimum and supremum, respectively: 
$$\mbox{\rm inf}({\mathcal A}(R,S))=\mbox{\rm max}\{\mbox{\rm scs}(M) 
\mid M\in {\mathcal A}(R,S)\},$$
$$\mbox{\rm sup}({\mathcal A}(R,S))=\mbox{\rm min}\{\mbox{\rm lcs}(M) 
\mid M\in {\mathcal A}(R,S)\}.$$

We henceforth focus on the case when row and column sums are constant. 
To this end, 
let 
$\Lambda_n^{x}$
 be the set of 
  $n\times n$ $(0,1)$-matrices  
with constant row and column sum $x$.

In fact, elements of 
$\Lambda_n^{x}$
 may also be thought of as {\em frequency squares} 
(sometimes {\em $F$-squares}). 
Let $n,\alpha,\lambda_1,\lambda_2,\dots ,\lambda_{\alpha}\in {\mathbb N}$  and  
$\sum_{i=1}^{\alpha} \lambda_i = n$.
A {\em frequency square} or {\em $F$-square} 
$F(n;\lambda_1,\lambda_2,\dots ,\lambda_{\alpha})$ 
of {\em order} $n$ is an $n\times n$ array on symbol set
$\{s_1,s_2,\dots ,s_{\alpha}\}$ such that each cell contains one symbol and 
symbol $s_i$ occurs precisely $\lambda_i$ times in each row and $\lambda_i$ times in each column.
Thus if we let $\alpha=2$, $s_1=1$ and $s_2=0$, the frequency squares  
$F(n;x,n-x)$ are precisely the elements of 
$\Lambda_n^{x}$.

Critical and defining sets of frequency squares were first studied in \cite{FSS}. 
Precise bounds for {\rm scs}$(\Lambda_{n}^{x})$ were obtained in \cite{FSS} for small values of $x$ and 
upper bounds for general $x$, including {\rm scs}$(\Lambda_{2m}^m)\leq m^2$. 
In \cite{Cav} it is shown that {\rm scs}$(\Lambda_{n}^x)\geq$min$\{x^2,(n-x)^2\}$, a corollary of which 
is:
\begin{theorem}
  {\rm scs}$(\Lambda_{2m}^m)= m^2$.
\label{smally} 
\end{theorem}

The following theorem was shown via constructing specific $(0,1)$-matrices which do not possess small defining sets. 
\begin{theorem} {\rm (\cite{CR})} 
If $m$ is a power of two, {\rm inf}$(\Lambda_{2m}^m)\geq 2m^2-O(m^{7/4})$. 
\label{mainy}
\end{theorem}
This result is currently being generalized to arbitrary $m$ as part of a more general result in \cite{BLW}.
Given that it is easy to show that any element of $\Lambda_{2m}^m$ has a defining set (and hence critical set) of size at most $2m^2$ (simply take all the $1$'s), this result is in some sense asympotically optimal. 
The analogous question has been considered for Latin squares in \cite{GHM}, where it is shown that every Latin square of order $n$ has a defining set of size at most $n^2-\frac{\sqrt{\pi}}{2}n^{9/6}$ and that for each $n$ there exists a Latin square $L$ with no defining set of size less than $n^2-(e+o(1))n^{10/6}$.   

In this paper we first find a new way of describing a critical set in any $(0,1)$-matrix (Theorem \ref{handier}). 
We then find asymptotically tight bounds for the largest critical set size:
$$3m^2-4m+2 \leq {\rm lcs}(\Lambda_{2m}^m)\leq 3m^2-2m$$
(Corollary \ref{larger}) 
as well as showing the existence of critical sets of each size between $m^2$ and $3m^2-4m+2$  
 (Theorem \ref{spek}). 
Finally, in Sections 4 and 5 we turn our attention to the supremum case, i.e. 
the function 
 {\rm sup}$(\Lambda_{2m}^m)$. We show that
$$ \lceil (3m^2-2m+1)/2\rceil \leq {\rm sup}(\Lambda_{2m}^m)\leq  2m^2-m,$$
 with the lower bound given by Theorem \ref{maineee} and the upper bound given by Theorem \ref{upsup}.

\section{Trades and critical sets in $(0,1)$-matrices}

In this section we list the theory from \cite{Cav} which is relevant to our paper, extending this theory with a new classification of critical sets in $(0,1)$-matrices. 
We define a {\em trade} to be a non-empty partial $(0,1)$-matrix $T$ 
 such that there exists a 
{\em disjoint mate} $T'$ such that:
\begin{itemize}
\item $T_{ij}$ is empty if and only if $T'_{ij}$ is empty;
\item if $T_{ij}$ is non-empty, then $T_{ij}\neq T'_{ij}$;
\item if $1$ appears precisely $k$ times in a row $r$ (column $c$) of $T$, then 
$1$ also appears $k$ times in row $r$ (column $c$) of $T'$;  
\item if $0$ appears precisely $k$ times in a row $r$ (column $c$) of $T$, then 
$0$ also appears $k$ times in row $r$ (column $c$) of $T'$;  
\end{itemize}

We can analyse the properties of critical sets of $(0,1)$-matrices via the trade structure of $(0,1)$-matrices. 
We say that a trade $T$ is a {\em cycle} 
if each row and each column of $T$ contains either $0$ or $2$ non-empty positions.  
A cycle is analogous to what Brualdi calls a
{\em minimal balanced matrix} (\cite{Bru}) and hence we have the following.  
\begin{theorem}{\rm (Lemma $3.2.1$ of \cite{Bru})} 
Any trade $T$ in a $(0,1)$-matrix is a union of disjoint cycles. 
\label{sii}
\end{theorem}

\begin{lemma}
A partial $(0,1)$-matrix $D$ is a defining set for a $(0,1)$-matrix $M$ if and only if $D\subseteq M$ and 
$|D\cap T|\geq 1$ for every cycle $T\subseteq M$.
\label{itsdef}
\end{lemma}

Since a critical set is a minimal defining set, we have the following.

\begin{corollary}
A partial $(0,1)$-matrix $D$ is a critical set for a $(0,1)$-matrix $M$ if and only if it is a defining set for $M$ and  
 for each $(i,j,D_{ij})\in D$, there exists a cycle $T\in M$ 
such that $T\cap D=\{(i,j,D_{ij})\}$. 
\label{itscri}
\end{corollary}

Next we give a new classification of critical sets in $(0,1)$-matrices that will be useful for our purposes. 
We embed any element $M$ of $\Lambda_{2m}^m$ in the Euclidian plane with the top left-hand corner in the origin. Specifically, 
 cell $(i,j)$ of $M$ becomes the unit square bordered by 
 the lines $x=j-1$, $x=j$, $y=-i$ and $y=1-i$. 
A {\em South East walk} in $M$ is then the walk induced by a sequence of points $w_0=(0,0)$, $w_1, w_2, \dots , w_{4m}=(2m,-2m)$, 
where either $w_{i+1}=w_{i}-(0,1)$ (a step {\em South}) or $w_{i+1}=w_i+(1,0)$ (a step {\em East}), for each $i$, $0\leq i<4m$.
Observe that $W$ begins at the top-left corner of $M$ and finishes at the bottom-right corner of $M$.

\begin{example}
Figure $\ref{tomato}$ gives an example of a South-East walk in a matrix from $\Lambda_{3}^6$. 
With respect to this walk, 
$$\begin{array}{c}
 (w_0,w_1,\dots ,w_{12})  = ((0,0),(1,0),(1,-1),(2,-1),(2,-2),(3,-2),(3,-3), \\
(4,-3),(4,-4),(5,-4),(5,-5),(5,-6),(6,-6)), \\
(P_0,P_1,\dots ,P_{11})=((0,0),(1,0),(1,-1),(2,-1),(2,-2),(3,-2),(3,-3), \\
(4,-3),(4,-4),(5,-4)(5,-6),(6,-6)),
\end{array}$$ 
$L=5$ and $L'=6$. By Theorem $\ref{handier}$, the elements in italics form a critical set in this matrix of size $14$. 
\begin{figure}
$$\begin{array}{|c!{\vrule width1.6pt}c
!{\vrule width1.6pt}c!{\vrule width1.6pt}c!{\vrule width1.6pt}c!{\vrule width1.6pt}c!{\vrule width2.4pt}}
\Xcline{1-1}{2.4pt}\Xcline{2-6}{0.8pt}
1 & \mcol{0} & \mcol{\it 1} & \mcol{\it 1} &\mcol{0} & \mcol{0}  \\
\Xcline{1-1}{0.8pt}\Xcline{2-2}{2.4pt}\Xcline{3-6}{0.8pt}
\mcol{1} & 1 & \mcol{0} & \mcol{0} &\mcol{0} & \mcol{\it 1}    \\
\Xcline{1-2}{0.8pt}\Xcline{3-3}{2.4pt}\Xcline{4-6}{0.8pt}
\mcol{\it 0} & \mcol{\it 0} & 1 & \mcol{0} &\mcol{\it 1} & \mcol{\it 1}  \\
\Xcline{1-3}{0.8pt}\Xcline{4-4}{2.4pt}\Xcline{5-6}{0.8pt}
\mcol{\it 0} & \mcol{1} & \mcol{\it 0} & 1 &\mcol{0} & \mcol{\it 1}  \\
\Xcline{1-4}{0.8pt}\Xcline{5-5}{2.4pt}\Xcline{6-6}{0.8pt}
\mcol{1} & \mcol{\it 0} & \mcol{1} & \mcol{\it 0} &1 & \mcol{0}   \\
\Xcline{1-6}{0.8pt}
\mcol{\it 0} & \mcol{1} & \mcol{\it 0} & \mcol{1} & 1 & \mcol{0}   \\
\Xcline{1-5}{0.8pt}\Xcline{6-6}{2.4pt}
\end{array}$$
\caption{A critical set of size $14$ in an element of $\Lambda_{6}^{3}$}
\label{tomato}
\end{figure}
\label{exm1}
\end{example}

The following theorem was shown in \cite{CR}.
\begin{theorem}
The set $D$ is a defining set of a $(0,1)$-matrix $M$ if and only if $D\subset M$ and the rows and columns 
 can be rearranged so that there exists a South-East Walk $W$  
 with only $1$'s below $W$ and only $0$'s above $W$, within the cells of $M\setminus D$.
\label{handy}
\end{theorem}

Next, we work towards a refinement of the previous theorem to give an equivalent definition of a critical set in a $(0,1)$-matrix. 
To this end, let $C$ be a critical set of a $(0,1)$-matrix $M$. From Theorem \ref{handy}, we may assume that the rows and columns of $M\setminus C$ are rearranged so that there exists a South-East walk $W$ with only $1$'s below $W$ and only $0$'s above $W$.
We rearrange the rows and columns of $M$ identically. 
We assume, without loss of generality, that $W$ includes the point $(1,0)$ (otherwise, swap $0$'s with $1$'s and take the transpose of $M$ to obtain an equivalent walk). 

We wish to define the {\em corners} of the walk $W$. These are the points where the walks changes direction from South to East (or vice versa).  
To this end, there is a uniquely-defined list of points $P_i=(r_i,c_i)$, $i\geq 0$, each on $W$ and with integer 
coordinates, such that $P_0=(0,0)$ and $P_{\ell}=(m,-m)$ and:
\begin{itemize}
\item If $k$ is odd, $r_k=r_{k-1}$ and $c_k>c_{k-1}$;
\item if $k$ is even, $c_k=c_{k-1}$ and $r_k>r_{k-1}$. 
\end{itemize} 

If the last step is South, then $P_{2\ell}=(r_{\ell},c_{\ell})$ for some $\ell$. In this case we define $L=L'=\ell$. Otherwise 
the last step is East and $P_{2\ell-1}=(r_{\ell-1},c_{\ell})$ for some $\ell$; here we define $L=\ell-1$ and $L'=\ell$. 
 For each $1\leq i\leq L$ and $1\leq j\leq L'$,  
let $L_{i,j}$ be the {\em block} of cells defined as follows:
$$L_{i,j}:=\{(r,c)\mid r_i<r \leq r_{i+1}, c_j < c\leq c_{j+1}\}.$$
The blocks defined above are simply the partition of $M$ induced by the corners of the walk. 

\begin{lemma}
If $C$ is a critical set, every cell in blocks of the form $L_{i,i}$ ($i\leq L$) contains $1$ and every cell 
in blocks of the form $L_{i,i+1}$ ($i\leq L'-1$) contains $0$. 
\end{lemma}

\begin{proof}
Suppose, for the sake of contradiction, that cell $(r,c)$ contains $0$ and belongs to block $L_{i,i}$ for some $i$.
Swap row $r$ with $r_i+1$ and column $c$ with column $c_i+1$. 
Observe that these swaps do not change any of the properties of $W$ with respect to $C$ - in particular $C$ is still a defining set 
as in Theorem \ref{handy}. 

Next, modify the walk $W$ to obtain the unique South-East walk $W'$ such cell $(r,c)$ is {\em above} $W'$ but every cell not equal to $(r,c)$ is below $W'$ if and only if it is below $W$.
Then $W'$ is still a South-East walk and thus implies the existence of a defining set $D'$ as in Theorem \ref{handy}. 
But $D'\subset C$, so $C$ is not a minimal defining set, a contradiction. 

The case when a block of the form $L_{i,i+1}$ contains $0$ is similar. 
\end{proof}

We have the following. 
\begin{theorem}
A set $C$ is a critical set of a $(0,1)$-matrix $M$ if and only if $C\subset M$ and the rows and columns 
of $M$ (and $C$) can be rearranged so that there exists a South-East Walk $W$ such that:
\begin{itemize}
\item The walk $W$ begins in the top-left corner of $M$ and finishes in the bottom-right corner of $M$; 
\item Within $M\setminus C$, there are only $1$'s below $W$ and only $0$'s above $W$; 
\item Within $M$, each cell in a block bordering $W$ from below contains $1$;
\item Within $M$, each cell in a block bordering $W$ from above contains $0$. 
\end{itemize}
\label{handier}
\end{theorem}

\begin{proof}
From the discussion above it remains to show the final claim that these conditions are sufficient.  From Theorem \ref{handy}, such a subset $C$ is a defining set of $M$, so from Corollary \ref{itscri} and Theorem \ref{sii} it is sufficient to show that for each element of $C$ there exists a  cycle which intersects $C$ only at that element.
Let $(r,c,0)\in C$. Then $(r,c,0)\in L_{a,b}$ for some block such that $a<b$. 
Then there exists a cycle:
$$\{(r,c,0),(r_a,c,1),(r_a,c_{a+1},0),(r_{a+1},c_{a+1},1),\dots ,(r_{b-1},c_b,0),(r,c_b,1)\},$$
  where $(r_i,c_i)$ is any cell in block $L_{i,i}$ and $(r_i,c_{i+1})$ is any cell in block $L_{i,i+1}$, for each $i$ where these cells are in the cycle.  
The case when $(r,c,1)\in C$ is similar. 
\end{proof}

Theorems \ref{handy} and \ref{handier} imply the following result, which was first proved in \cite{Cav}. 

\begin{theorem}
The complement of a minimal defining set in a $(0,1)$-matrix is a defining set.
\label{compli}
\end{theorem}

In fact, we can improve on this a little. 

\begin{theorem}
Let $C$ be a critical set of a matrix $M$ in $\Lambda_{2m}^m$. 
Then there is a defining set $D$ for $M$ such that $D\subset M\setminus C$ and $|D|\leq 4m^2-2m-|C|$. 
\end{theorem}

\begin{proof}
Let $C$ be a critical set of a matrix $M$ in $\Lambda_{2m}^m$. 
Then the rows and columns of $M$ can be arranged so that there exists a South-East walk satisfying the conditions of  Theorem \ref{handier}. 
Let $\{L_{i,j}\mid 1\leq i\leq L, 1\leq i\leq L'\}$ be the set of blocks with respect to the walk $W$.
Next, let $W'$ be the unique South-East walk such that:
\begin{itemize} 
\item If $i>j$, block $L_{i,j}$ is below $W'$; and 
\item if $i\leq j$, block $L_{i,j}$ is above $W'$. 
\end{itemize}
Since $W'$ is a South-East walk, by Theorem \ref{handy}, the set of cells below $W'$ containing $1$ and the set of cells above $W'$ containing $0$ form a defining set $D$ of $M$. 
(Note that we are actually applying an equivalent version of Theorem \ref{handy} with $1$'s and $0$'s swapped.) 
Moreover, by construction $D$ avoids all cells of $C$ and all cells from the main diagonal of blocks (since by Theorem \ref{handier}, such blocks contain no $0$'s). Indeed,   
$|D|=4m^2-|C|-\sum_{i=1}^L |L_{i,i}|.$ 
But each row contains at least one element in the main diagonal of blocks, so 
$|D|\leq 4m^2-2m-|C|$. 
\end{proof}

\begin{corollary}
If $C$ is a critical set in a matrix $M$ in $\Lambda_{2m}^m$, then 
$|C|\leq 3m^2-2m$. 
\label{larger}
\end{corollary}

\begin{proof}
For the sake of contradiction, suppose there exists a critical set $C$ in a matrix $M$ in $\Lambda_{2m}^m$ and that $|C|>3m^2-2m$. 
Then, by the previous theorem, there exists a defining set $D$ in $M$ (and thus a minimal defining set, i.e. a critical set) of size less than $m^2$. 
 This contradicts Theorem \ref{smally}.
\end{proof}

\section{The spectrum of critical set sizes}

In this section we show that for each $m\geq 1$ and each $k$ such that $m^2\leq k\leq 3m^2-4m+2$, there exists a critical set of size $k$ in some matrix from $\Lambda_{2m}^m$. 
From Corollary \ref{larger}, lcs($\Lambda_{2m}^m$)$\leq 3m^2-2m$, so this result completes the spectrum with less than $2m$ possible exceptions. We conjecture that there are no exceptions and lcs$(\Lambda_{2m}^m)=3m^2-4m+2$.  

We first show that lcs$(\Lambda_{2m}^m)\geq 3m^2-4m+2$ by showing the existence of a critical set of such size for each $m$.  
By observation, such a critical set exists for $m=1$. Otherwise, let $m\geq 2$. 
We define a $(0,1)$-matrix $X_{2m}$ as follows. 
Cell $(i,j)$ contains:
\begin{itemize} 
\item $0$ if $i-j\equiv k$ (mod $2n$) where $k\in \{1,2,\dots ,m-1,2m-1\}$;
\item  $1$ otherwise. 
\end{itemize}
It is clear that $X_{2m}\in \Lambda_{2m}^m$. 
Now, let $W$ be the unique South-East walk from the top left-hand corner of $X_{2m}$ to the bottom right-hand corner of $X_{2m}$ which 
borders the main diagonal from above. 

\begin{figure}
$$
\begin{array}{|c!{\vrule width1.6pt}c
!{\vrule width1.6pt}c!{\vrule width1.6pt}c!{\vrule width1.6pt}c!{\vrule width1.6pt}c!{\vrule width1.6pt}c!
{\vrule width1.6pt}c!{\vrule width1.6pt}}
\Xcline{1-1}{1.6pt}\Xcline{2-8}{0.8pt}
1 & \multicolumn{1}{c|}{0} &\multicolumn{1}{|c|}{\it 1} & \multicolumn{1}{|c|}{\it 1} & \multicolumn{1}{|c|}{\it 1} & \multicolumn{1}{|c|}{0} & \multicolumn{1}{|c|}{0} & \multicolumn{1}{|c|}{0} \\
\Xcline{1-1}{0.8pt}\Xcline{2-2}{1.6pt}\Xcline{3-8}{0.8pt}
\multicolumn{1}{|c|}{\it 0} &1 & \multicolumn{1}{c|}{0} & \multicolumn{1}{|c|}{\it 1} & 
\multicolumn{1}{|c|}{\it 1} & \multicolumn{1}{|c|}{\it 1} & 
\multicolumn{1}{|c|}{0} & \multicolumn{1}{|c|}{0}  \\
\Xcline{1-2}{0.8pt}\Xcline{3-3}{1.6pt}\Xcline{4-8}{0.8pt}
\multicolumn{1}{|c|}{\it 0} &
\multicolumn{1}{|c|}{\it 0} &
1 & \multicolumn{1}{c|}{0} & \multicolumn{1}{|c|}{\it 1} & 
\multicolumn{1}{|c|}{\it 1} & \multicolumn{1}{|c|}{\it 1} & 
\multicolumn{1}{|c|}{0}   \\
\Xcline{1-3}{0.8pt}\Xcline{4-4}{1.6pt}\Xcline{5-8}{0.8pt}
\multicolumn{1}{|c|}{\it 0} &
\multicolumn{1}{|c|}{\it 0} &
\multicolumn{1}{|c|}{\it 0} & 1 & \multicolumn{1}{c|}{0} & \multicolumn{1}{|c|}{\it 1} & 
\multicolumn{1}{|c|}{\it 1} & \multicolumn{1}{|c|}{\it 1}   \\
\Xcline{1-4}{0.8pt}\Xcline{5-5}{1.6pt}\Xcline{6-8}{0.8pt}
\multicolumn{1}{|c|}{1} &
\multicolumn{1}{|c|}{\it 0} &
\multicolumn{1}{|c|}{\it 0} &
\multicolumn{1}{|c|}{\it 0} & 1 & \multicolumn{1}{c|}{0} & \multicolumn{1}{|c|}{\it 1} & 
\multicolumn{1}{|c|}{\it 1}    \\
\Xcline{1-5}{0.8pt}\Xcline{6-6}{1.6pt}\Xcline{7-8}{0.8pt}
\multicolumn{1}{|c|}{1} &
\multicolumn{1}{|c|}{1} &
\multicolumn{1}{|c|}{\it 0} &
\multicolumn{1}{|c|}{\it 0} &
\multicolumn{1}{|c|}{\it 0} & 1 & \multicolumn{1}{c|}{0} & \multicolumn{1}{|c|}{\it 1}   \\
\Xcline{1-6}{0.8pt}\Xcline{7-7}{1.6pt}\Xcline{8-8}{0.8pt}
\multicolumn{1}{|c|}{1} &
\multicolumn{1}{|c|}{1} &
\multicolumn{1}{|c|}{1} &
\multicolumn{1}{|c|}{\it 0} &
\multicolumn{1}{|c|}{\it 0} &
\multicolumn{1}{|c|}{\it 0} & 1 & \multicolumn{1}{c|}{0}  \\
\Xcline{1-7}{0.8pt}\Xcline{8-8}{1.6pt}
\multicolumn{1}{|c|}{\it 0} &
\multicolumn{1}{|c|}{1} &
\multicolumn{1}{|c|}{1} &
\multicolumn{1}{|c|}{1} &
\multicolumn{1}{|c|}{\it 0} &
\multicolumn{1}{|c|}{\it 0} &
\multicolumn{1}{|c|}{\it 0} & 1  \\
\hline
\multicolumn{8}{c}{X_8} 
\end{array}$$
\caption{The matrix $X_8$ with elements of a critical set shown in italics.} 
\label{ookii}
\end{figure}

By Theorem \ref{handier}, this defines a critical set $C$ which consists of each $0$ below $W$ and each $1$ above $W$. 
 Therefore we have the following. 
\begin{theorem}
Let $m\geq 1$. There exists a critical set in $X_{2m}\in \Lambda_{2m}^m$ of size $3m^2-4m+2$. 
\label{maxy}
\end{theorem}

See Figure \ref{ookii} for a demonstration of the previous theorem when $m=4$.  
From Corollary \ref{larger} we have: 

\begin{corollary}
$3m^2-4m+2\leq lcs(\Lambda_{2m}^m)\leq 3m^2-2m$. 
\end{corollary}

Next, we fill the lower part of the spectrum. 

\begin{lemma}
Let $m\geq 1$. For each $k$,  $m^2\leq k\leq m^2+(m-1)^2$, there exists a critical set of size $k$ in some matrix from $\Lambda_{2m}^m$. 
\label{lower}
\end{lemma}

\begin{proof}
We define a matrix $M(k)\in \Lambda_{2m}^m$ as follows. 
Let $k-m^2=\alpha(m-1)+\beta$, where $\alpha\geq 0$ and $0\leq \beta<m-1$. 
Let $W$ be the unique South-East walk including the points:
 $$\begin{array}{l} 
(0,0),(m,0), (m,\alpha-m), (2m-1,\alpha-m), (2m-1,-m),\\
 (2m,-m),(2m,-2m) \mbox{\ (if $\beta=0$);}
\end{array}$$
 otherwise $\beta>1$ and 
 let $W$ be the unique South-East walk including the points:
  $$\begin{array}{l}
(0,0),(m,0), (m,\alpha+1-m),(m+\beta,\alpha+1-m),(m+\beta,\alpha-m),\\
(2m-1,\alpha-m), (2m-1,-m), (2m,-m),(2m,-2m).
\end{array}$$

\begin{figure}
$$\begin{tabular}{c c c c c c c c}
& $\leftarrow$ & $\beta$ & $\rightarrow$ &
\multicolumn{2}{c}{} & \hz  \\
\Xcline{2-7}{0.8pt}
\multicolumn{1}{c}{} &\vmcol{2pt}{0} & \vmcol{0pt}{0} & \vmcol{0pt}{0} & \vmcol{0pt}{\hz} & \vmcol{0pt}{\hz} & \vmcol{0pt}{\hz}\\ 
\Xcline{2-7}{0.8pt}
\multicolumn{1}{c}{} &\vmcol{2pt}{0} & \vmcol{0pt}{0} & \vmcol{0pt}{0} & \vmcol{0pt}{\hz} & \vmcol{0pt}{\hz} & \vmcol{0pt}{\hz}\\ 
\Xcline{2-4}{2.4pt}\Xcline{4-7}{0.8pt}
\multicolumn{1}{c}{} &\vmcol{0.8pt}{1} & \vmcol{0pt}{1} & \vmcol{0pt}{1} & \vmcol{2.0pt}{0} & \vmcol{0pt}{0} & \vmcol{0pt}{\hz} \\
\Xcline{2-4}{0.8pt}\Xcline{5-6}{2.4pt}\Xcline{7-7}{0.8pt} 
\multicolumn{1}{c}{$\uparrow$} &\vmcol{0.8pt}{\hz} & \vmcol{0pt}{\hz} & 
 \vmcol{0pt}{\hz}  &  \vmcol{0pt}{1} &  \vmcol{0pt}{1} & 
\vmcol{2.0pt}{0} \\
\Xcline{2-7}{0.8pt}
\multicolumn{1}{c}{$\alpha$} &\vmcol{0.8pt}{\hz} & \vmcol{0pt}{\hz} & 
 \vmcol{0pt}{\hz}  &  \vmcol{0pt}{1} &  \vmcol{0pt}{1} & 
\vmcol{2.0pt}{0} \\
\Xcline{2-7}{0.8pt}
\multicolumn{1}{c}{$\downarrow$} &\vmcol{0.8pt}{\hz} & \vmcol{0pt}{\hz} & 
 \vmcol{0pt}{\hz}  &  \vmcol{0pt}{1} &  \vmcol{0pt}{1} & 
\vmcol{2.0pt}{0} \\
\Xcline{2-6}{0.8pt} \Xcline{7-7}{2.4pt}
\end{tabular}
$$
\label{fig1}
\caption{The walk $W$ within the quadrant $Q$.}
\end{figure}

Let $Q$ be the quadrant of cells in $M(k)$ bordered by points $(m,0)$, $(m,-m)$, $(2m,-m)$ and $(2m,0)$, 
Place $1$ in each cell of $Q$ below $W$ and $0$ in each cell of $Q$ above $W$. 
Next, for each cell $(i,j)\in Q$ containing entry $e$, let cell $(i-m,j-m)$ contain entry $e$, cell $(i-m,j)$ contain entry $1-e$ and 
cell $(i,j-m)$ contain entry $1-e$. We illustrate the walk $W$ within $Q$ (and its induced blocks) in Figure \ref{fig1}.  A complete example of $M(k)$ is given in Figure \ref{tryagain}. 

Clearly $M(k)$ thus defined is an element of $\Lambda_{2m}^{m}$. Observe also that the walk $W$ defines a critical set $C$ as in Theorem \ref{handier}. Such a critical set consists of all occurrences $0$ below the walk $W$. Thus every $0$ in the first $m$ columns belongs to $C$ (a total of $m^2$), no $0$ from $Q$ occurs in $C$ and each of the $\alpha(m-1)+\beta$ $0$'s in the quadrant below $Q$ belongs to $C$. Thus $C$  has size $k$, as required.   
\end{proof}

\begin{figure}
$$
\begin{array}{|c!{\vrule width1.6pt}c
!{\vrule width1.6pt}c!{\vrule width1.6pt}c!{\vrule width1.6pt}c!{\vrule width1.6pt}c!{\vrule width1.6pt}c!
{\vrule width1.6pt}c!{\vrule width2.4pt}}
\Xcline{1-4}{2.4pt}\Xcline{5-8}{0.8pt}
\mcol{1} & \mcol{1} & \mcol{1} &\mcol{1} & \vmcol{2.4pt}{0} &\mcol{0} &\mcol{0} & \mcol{0}  \\
\Xcline{1-8}{0.8pt}
\mcol{1} & \mcol{1} & \mcol{1} &\mcol{1} & \vmcol{2.4pt}{0} &\mcol{0} &\mcol{0} & \mcol{0}  \\
\Xcline{1-4}{0.8pt}\Xcline{5-5}{2.4pt}\Xcline{6-8}{0.8pt}
\mcol{\it 0} & \mcol{1} & \mcol{1} & \mcol{1} & \mcol{1} & \vmcol{2.4pt}{0} &\mcol{0} & \mcol{0}   \\
\Xcline{1-5}{0.8pt}\Xcline{6-7}{2.4pt}\Xcline{8-8}{0.8pt}
\mcol{\it 0} & \mcol{\it 0} & \mcol{\it 0} & 
\mcol{1} & \mcol{1} & \mcol{1} & \mcol{1} & \vmcol{2.4pt}{0}   \\
\Xcline{1-7}{0.8pt}\Xcline{8-8}{2.4pt}
\mcol{\it 0} & \mcol{\it 0} & \mcol{\it 0} & \mcol{\it 0} &  
\mcol{1} & \mcol{1} & \mcol{1} & 1    \\
\Xcline{1-8}{0.8pt}
\mcol{\it 0} & \mcol{\it 0} & \mcol{\it 0} & \mcol{\it 0} &  
\mcol{1} & \mcol{1} & \mcol{1} & 1    \\
\Xcline{1-8}{0.8pt}
\mcol{1} & \mcol{\it 0} & \mcol{\it 0} & \mcol{\it 0} & \mcol{\it 0} &  
\mcol{1} & \mcol{1} & 1    \\
\Xcline{1-8}{0.8pt}
\mcol{1} & \mcol{1} & \mcol{1} & \mcol{\it 0} & \mcol{\it 0} & \mcol{\it 0} & \mcol{\it 0} &  
1    \\
\Xcline{1-8}{0.8pt}
\end{array}
$$
\caption{$M(4)$ with $m=20$, $\alpha=\beta=1$; the elements of the critical set are shown in italics.} 
\label{tryagain}
\end{figure}

\begin{lemma}
Let $m\geq 4$. 
For each $k$,  $2m^2-5m+15\leq k\leq 3m^2-4m+2$, there exists a critical set of size $k$ in some matrix from $\Lambda_{2m}^m$. 
\label{upper}
\end{lemma}

\begin{proof}
Let $I$ be the following set of cells in $X_{2m}$:
$$I:=\{(i,j)\mid m<i\leq 2m, 1\leq j<m-2, m\leq i-j<2m-1\}.$$
Observe that $|I|=m(m+1)/2-7$. 
For each cell $(i,j)\in I$, define a trade 
$T(i,j)\subset X_{2m}$
on the cells: 
$$\{(i,j),(i-m+1,j),(i,i-j),(i-m+1,i-j)\}.$$

In Figure \ref{fig3}, the elements of $I$ are shown in bold for the case $m=5$. 
 \begin{figure}
$$\begin{array}
{|c!{\vrule width1.6pt}c
!{\vrule width1.6pt}c!{\vrule width1.6pt}c!{\vrule width1.6pt}c!{\vrule width1.6pt}c!{\vrule width1.6pt}c!
{\vrule width1.6pt}c!{\vrule width1.6pt}c!{\vrule width 1.6pt}c!{\vrule width 1.6pt}}
\Xcline{1-1}{2.4pt}\Xcline{2-10}{0.8pt}
1 & \mcol{0} & \mcol{\it 1} & \mcol{\it 1} &\mcol{\it 1} &\mcol{\it 1} & \mcol{0} &\mcol{0} & \mcol{0} & \mcol{0} \\
\Xcline{1-1}{0.8pt}\Xcline{2-2}{2.4pt}\Xcline{3-10}{0.8pt}
\mcol{\it 0_A} & 1 & \mcol{0} & \mcol{\it 1} &\mcol{\it 1_A} &\mcol{\it 1} & \mcol{\it 1} &\mcol{0} & \mcol{0} & \mcol{0} \\
\Xcline{1-2}{0.8pt}\Xcline{3-3}{2.4pt}\Xcline{4-10}{0.8pt}
\mcol{\it 0_B} & \mcol{\it 0_E} & 1 & \mcol{0} & \mcol{\it 1_E} &\mcol{\it 1_B} &\mcol{\it 1} & \mcol{\it 1} &\mcol{0} & \mcol{0} \\
\Xcline{1-3}{0.8pt}\Xcline{4-4}{2.4pt}\Xcline{5-10}{0.8pt}
\mcol{\it 0_C} & \mcol{\it 0_F} & \mcol{\it 0} & 1 & \mcol{0} & \mcol{\it 1_F} &\mcol{\it 1_C} &\mcol{\it 1} & \mcol{\it 1} &\mcol{0} \\
\Xcline{1-4}{0.8pt}\Xcline{5-5}{2.4pt}\Xcline{6-10}{0.8pt}
\mcol{\it 0_D} & \mcol{\it 0_G} & \mcol{\it 0} & \mcol{\it 0} & 1 & \mcol{0} & \mcol{\it 1_G} &\mcol{\it 1_D} &\mcol{\it 1} & \mcol{\it 1} \\
\Xcline{1-5}{0.8pt}\Xcline{6-6}{2.4pt}\Xcline{7-10}{0.8pt}
\mcol{\bf 1_A} & \mcol{\it 0_H} & \mcol{\it 0} & \mcol{\it 0} & \mcol{\it 0_A} & 1 & \mcol{0} & \mcol{\it 1_H} &\mcol{\it 1} &\mcol{\it 1}  \\
\Xcline{1-6}{0.8pt}\Xcline{7-7}{2.4pt}\Xcline{8-10}{0.8pt}
\mcol{\bf 1_B} & \mcol{\bf 1_E} & \mcol{\it 0} & \mcol{\it 0} & \mcol{\it 0_E} & \mcol{\it 0_B} & 1 & \mcol{0} & \mcol{\it 1} &\mcol{\it 1}  \\
\Xcline{1-7}{0.8pt}\Xcline{8-8}{2.4pt}\Xcline{9-10}{0.8pt}
\mcol{\bf 1_C} & \mcol{\bf 1_F} & \mcol{1} & \mcol{\it 0} & \mcol{\it 0} & \mcol{\it 0_F} & \mcol{\it 0_C} & 1 & \mcol{0} & \mcol{\it 1}  \\
\Xcline{1-8}{0.8pt}\Xcline{9-9}{2.4pt}\Xcline{10-10}{0.8pt}
\mcol{\bf 1_D} & \mcol{\bf 1_G} & \mcol{1} & \mcol{1} & \mcol{\it 0} & \mcol{\it 0} & \mcol{\it 0_G} & \mcol{\it 0_D} & 1 & \mcol{0}   \\
\Xcline{1-9}{0.8pt} \Xcline{10-10}{2.4pt}
\mcol{\it 0} & \mcol{\bf 1_H} & \mcol{1} & \mcol{1} & \mcol{1} & \mcol{\it 0} & \mcol{\it 0} & \mcol{\it 0_H} & \mcol{\it 0} & 1 \\
\Xcline{1-10}{0.8pt}
\end{array}$$
\caption{The elements of $I$ in the case $m=5$; trades are shown by common subscripts.} 
\label{fig3}
\end{figure}

We claim that the set $\{T(i,j)\mid (i,j)\in I\}$ forms a set of $|I|$ disjoint trades and that only one element of $T(i,j)$ (containing $1$) lies above the walk $W$ for each $(i,j)\in I$.  
It suffices to show disjointness and that for each $(i,j)\in I$:

\begin{itemize}
\item Cells $(i,j)$ and $(i-(m-1),i-j)$ contain 1;
\item Cells $(i,i-j)$ and $(i-(m-1),j)$ contain 0;
\item Cell $(i-(m-1),j)$ lies above $W$ and each other cell lies below.
\end{itemize}

Since $m \leqslant i-j < 2m - 1$, by the definition of $X_{2m}$, cell $(i,j)$ contains $1$ whenever $(i,j)\in I$. We also note that $i-j > 0$ for each $(i,j)\in I$, so each such cell $(i,j)$ lies below $W$.

Next consider the cell $(i-(m-1),i-j)$ where $(i,j)\in I$. Then 
$1\leq j\leq m-3$ implies that 
$$ -m \leqslant (i-(m-1))-(i-j)\leqslant -4.$$
Thus from the definition of $X_m$, cell $(i-(m-1),i-j)$ contains $1$.
Since $2\leq i-(m-1)\leq m$, such cells lie above the main diagonal (and thus above $W$).  

Next, since $m\leqslant i-j\leq 2m-2$, we have:
$$-m+1\leq j-(i-(m-1))\leq -1,$$ so each cell of the form $(i-(m-1),j)$ contains $0$. 
 Since $1\leq j\leq m-3$, such cells also lie below the main diagonal. 

Finally, we check that cells of the form $(i,i-j)$ always contain 0. This follows from the fact that 
$1\leq j\leq m-3$. Since $m+1\leq i\leq 2m$, such cells also lie below the main diagonal. 

We now check the disjointness of the trades. 
As cells of the form $(i-(m-1),i-j)$ lie above $W$ and cells of the form $(i,j)$ lie below $W$, there is no intersection in cells containing 0. (It is straightforward to check that cells of the same form are distinct.) 
Finally since
$j\leq m-1< i-j$, cells of the from 
$(i,i-j)$ and $(i-(m-1),j)$ are distinct. Again, it is straightforward to check the cells of the same such forms are distinct. 

Thus, replacing exactly $\alpha$ of these trades by their disjoint mates creates a critical set of size $3m^2-4m+2-2\alpha$. 
Since $|I|=m(m+1)/2-7$, this yields critical sets of size 
$3m^2-4m+2-2\alpha$ whenever $0\leq \alpha\leq m(m+1)/2-7$.

Define $Y_{2m}$  to be the element of $\Lambda_{2m}^m$ formed by swapping $0$ and $1$ in the cells
$(m-1,2m-1)$, $(m-1,2m)$, $(2m,2m-1)$ and $(2m,2m)$. 
The matrix $Y_{10}$ is given in Figure \ref{fig4}. Observe that $Y_{2m}$ has a critical set of size 
$3m^2-4m+1$ by making a small adjustment to our South East walk so that cell $(2m,2m)$ lies above the South East walk, with all other cells as before. 

\begin{figure}
$$\begin{array}
{|c!{\vrule width1.6pt}c
!{\vrule width1.6pt}c!{\vrule width1.6pt}c!{\vrule width1.6pt}c!{\vrule width1.6pt}c!{\vrule width1.6pt}c!
{\vrule width1.6pt}c!{\vrule width1.6pt}c!{\vrule width 1.6pt}c!{\vrule width 1.6pt}}
\Xcline{1-1}{2.4pt}\Xcline{2-10}{0.8pt}
1 & \mcol{0} & \mcol{\it 1} & \mcol{\it 1} &\mcol{\it 1} &\mcol{\it 1} & \mcol{0} &\mcol{0} & \mcol{0} & \mcol{0} \\
\Xcline{1-1}{0.8pt}\Xcline{2-2}{2.4pt}\Xcline{3-10}{0.8pt}
\mcol{\it 0} & 1 & \mcol{0} & \mcol{\it 1} &\mcol{\it 1} &\mcol{\it 1} & \mcol{\it 1} &\mcol{0} & \mcol{0} & \mcol{0} \\
\Xcline{1-2}{0.8pt}\Xcline{3-3}{2.4pt}\Xcline{4-10}{0.8pt}
\mcol{\it 0} & \mcol{\it 0} & 1 & \mcol{0} & \mcol{\it 1} &\mcol{\it 1} &\mcol{\it 1} & \mcol{\it 1} &\mcol{0} & \mcol{0} \\
\Xcline{1-3}{0.8pt}\Xcline{4-4}{2.4pt}\Xcline{5-10}{0.8pt}
\mcol{\it 0} & \mcol{\it 0} & \mcol{\it 0} & 1 & \mcol{0} & \mcol{\it 1} &\mcol{\it 1} &\mcol{\it 1} & \mcol{0} &\mcol{\it 1} \\
\Xcline{1-4}{0.8pt}\Xcline{5-5}{2.4pt}\Xcline{6-10}{0.8pt}
\mcol{\it 0} & \mcol{\it 0} & \mcol{\it 0} & \mcol{\it 0} & 1 & \mcol{0} & \mcol{\it 1} &\mcol{\it 1} &\mcol{\it 1} & \mcol{\it 1} \\
\Xcline{1-5}{0.8pt}\Xcline{6-6}{2.4pt}\Xcline{7-10}{0.8pt}
\mcol{1} & \mcol{\it 0} & \mcol{\it 0} & \mcol{\it 0} & \mcol{\it 0} & 1 & \mcol{0} & \mcol{\it 1} &\mcol{\it 1} &\mcol{\it 1}  \\
\Xcline{1-6}{0.8pt}\Xcline{7-7}{2.4pt}\Xcline{8-10}{0.8pt}
\mcol{1} & \mcol{1} & \mcol{\it 0} & \mcol{\it 0} & \mcol{\it 0} & \mcol{\it 0} & 1 & \mcol{0} & \mcol{\it 1} &\mcol{\it 1}  \\
\Xcline{1-7}{0.8pt}\Xcline{8-8}{2.4pt}\Xcline{9-10}{0.8pt}
\mcol{1} & \mcol{1} & \mcol{1} & \mcol{\it 0} & \mcol{\it 0} & \mcol{\it 0} & \mcol{\it 0} & 1 & \mcol{0} & \mcol{\it 1}  \\
\Xcline{1-8}{0.8pt}\Xcline{9-9}{2.4pt}\Xcline{10-10}{0.8pt}
\mcol{1} & \mcol{1} & \mcol{1} & \mcol{1} & \mcol{\it 0} & \mcol{\it 0} & \mcol{\it 0} & \mcol{\it 0} & 1 & \mcol{0}   \\
\Xcline{1-10}{0.8pt} 
\mcol{\it 0} & \mcol{1} & \mcol{1} & \mcol{1} & \mcol{1} & \mcol{\it 0} & \mcol{\it 0} & \mcol{\it 0} & 1 & \mcol{0}  \\
\Xcline{1-9}{0.8pt} \Xcline{10-10}{2.4pt}
\end{array}$$
\caption{The matrix $Y_{10}$}
\label{fig4}
\end{figure}

Moreover, the set of trades $\{T(i,j)\mid (i,j)\in I\}$ retains the same properties with respect to $Y_{2m}$. Thus 
there exists a critical set in $Y_{2m}$ of size $3m^2-4m+1-2\alpha$ whenever $0\leq \alpha\leq m(m+1)/2-7$. 
 We are done. 
\end{proof}

\begin{theorem}
Let $m\geq 1$. For each $k$,  $m^2\leq k\leq 3m^2-4m+2$, there exists a critical set of size $k$ in some matrix from $\Lambda_{2m}^m$. 
\label{spek}
\end{theorem}

\begin{proof}
For $m\leq 4$, the result is given by Theorem \ref{maxy}, Lemma \ref{lower}, Lemma \ref{upper} and Example \ref{exm1}, except for the following 
cases: $(m,k)\in\{(3,15),(3,16),(4,26)\}$. The case $m=3$ and $k=16$ can be found in $Y_6$ as in the previous proof; the remaining two cases are given in Figure \ref{filly}. 
Otherwise $m\geq 5$ and $2m^2-5m+15\leq m^2+(m-1)^2$ and the theorem follows from Lemmas \ref{lower} and \ref{upper}.  
\end{proof}

\begin{figure}
$$\begin{array}{|c!{\vrule width1.6pt}c
!{\vrule width1.6pt}c!{\vrule width1.6pt}c!{\vrule width1.6pt}c!{\vrule width1.6pt}c!{\vrule width2.4pt}}
\Xcline{1-1}{2.4pt}\Xcline{2-6}{0.8pt}
1 & \mcol{0} & \mcol{\it 1} & \mcol{\it 1} &\mcol{0} & \mcol{0}  \\
\Xcline{1-1}{0.8pt}\Xcline{2-2}{2.4pt}\Xcline{3-6}{0.8pt}
\mcol{1} & 1 & \mcol{0} & \mcol{0} &\mcol{\it 1} & \mcol{0}    \\
\Xcline{1-2}{0.8pt}\Xcline{3-3}{2.4pt}\Xcline{4-6}{0.8pt}
\mcol{\it 0} & \mcol{\it 0} & 1 & \mcol{0} &\mcol{\it 1} & \mcol{\it 1}  \\
\Xcline{1-3}{0.8pt}\Xcline{4-4}{2.4pt}\Xcline{5-6}{0.8pt}
\mcol{\it 0} & \mcol{1} & \mcol{\it 0} & 1 &\mcol{0} & \mcol{\it 1}  \\
\Xcline{1-4}{0.8pt}\Xcline{5-5}{2.4pt}\Xcline{6-6}{0.8pt}
\mcol{1} & \mcol{\it 0} & \mcol{1} & \mcol{\it 0} & 1 & \mcol{0}   \\
\Xcline{1-5}{0.8pt}\Xcline{6-6}{2.4pt}
 \mcol{\it 0} & \mcol{1} & \mcol{\it 0} & \mcol{1} & \mcol{\it 0} & 1   \\
\Xcline{1-6}{0.8pt}
\end{array} \label{$|C| = 14$}
\hphantom{rawr}
\begin{array}
{|c!{\vrule width1.6pt}c
!{\vrule width1.6pt}c!{\vrule width1.6pt}c!{\vrule width1.6pt}c!{\vrule width1.6pt}c!{\vrule width1.6pt}c!
{\vrule width1.6pt}c!{\vrule width2.4pt}}
\Xcline{1-1}{2.4pt}\Xcline{2-8}{0.8pt}
1 & \mcol{0} & \mcol{\it 1} &\mcol{\it 1} &\mcol{\it 1} & \mcol{0} &\mcol{0} & \mcol{0}  \\
\Xcline{1-1}{0.8pt}\Xcline{2-2}{2.4pt}\Xcline{3-8}{0.8pt}
\mcol{1} & 1 & \mcol{0} &\mcol{0} & \mcol{\it 1} &\mcol{\it 1} &\mcol{0} & \mcol{0} \\
\Xcline{1-2}{0.8pt}\Xcline{3-3}{2.4pt}\Xcline{4-8}{0.8pt}
\mcol{1} & \mcol{1} & 1 &\mcol{0} &\mcol{0} & \mcol{0} &\mcol{\it 1} & \mcol{0} \\
\Xcline{1-3}{0.8pt}\Xcline{4-4}{2.4pt}\Xcline{5-8}{0.8pt}
\mcol{1} & \mcol{\it 0} & \mcol{\it 0} & 1 &\mcol{0} & \mcol{0} &\mcol{\it 1} & \mcol{\it 1} \\
\Xcline{1-4}{0.8pt}\Xcline{5-5}{2.4pt}\Xcline{6-8}{0.8pt}
\mcol{\it 0} & \mcol{\it 0} & \mcol{\it 0} & \mcol{1} & 1 & \mcol{0} &\mcol{\it 1} & \mcol{\it 1} \\
\Xcline{1-5}{0.8pt}\Xcline{6-6}{2.4pt}\Xcline{7-8}{0.8pt}
\vmcol{0.8pt}{\it 0} & \mcol{1} & \mcol{\it 0} & \mcol{\it 0} & \mcol{1} & 1 &\mcol{0} & \mcol{\it 1} \\
\Xcline{1-6}{0.8pt}\Xcline{7-7}{2.4pt}\Xcline{8-8}{0.8pt}
\vmcol{0.8pt}{\it 0} & \mcol{1} & \mcol{1} & \mcol{\it 0} & \mcol{\it 0} & \mcol{1} & 1 & \mcol{0}\\
\Xcline{1-7}{0.8pt}\Xcline{8-8}{2.4pt}
\vmcol{0.8pt}{\it 0} & \mcol{\it 0} & \mcol{1} & \mcol{1} & \mcol{\it 0} & \mcol{1} & \mcol{\it 0} & 1 \\
\Xcline{1-8}{0.8pt}
\end{array}$$
\caption{Critical sets of size $15$ and $26$.}
\label{filly}
\end{figure}

\section{A lower bound on the supremum}

In this section we show that every element of $\Lambda_{2m}^m$ contains a critical set of size greater than $3m(m-1)/2$. 
Thus:
\begin{theorem}
{\rm sup}$(\Lambda_{2m}^m)\geq \lceil (3m^2-2m+1)/2\rceil $.  
\label{maineee}
\end{theorem}
  
In the following, 
 let $M\in \Lambda_{2m}^m$. By rearranging rows and columns, 
we may assume that cells $(i,0)$ and $(0,j)$ contain $1$ whenever $i\leq m$ or $j\leq m$. 
Let $R_1=C_1=\{1\}$, $R_2=C_2=\{2,3,\dots ,m\}$ and $R_3=C_3=\{m+1,m+2,\dots ,2m\}$.  
These sets give partitions of the rows and columns, so that we may define the subarray 
$M_{i,j}$ to be the intersection of the rows from $R_i$ and the columns from $C_j$. 

\begin{lemma}
Let $M$ be any element of $\Lambda_{2m}^m$. 
Then there exists a critical set $C_1$ of $M$ including:
\begin{itemize}
\item $(1,1,1)$;
\item no elements from $M_{1,2}\cup M_{1,3}\cup M_{2,1}\cup M_{3,1}$;
\item each $0$ from $M_{2,2}\cup M_{2,3}\cup M_{3,2}$;
\item no $1$'s from $M_{2,2}\cup M_{2,3}\cup M_{3,2}\cup M_{3,3}$. 
\end{itemize}
\label{ian1}
\end{lemma}

\begin{proof}
First, let $D$ be the subset of $M$ defined as above with the extra property that {\em every} $0$ from 
$M_{3,3}$ is included. 
We first show that $D$ is a defining set for $M$. 
Let $M'$ be any completion of $D$ to a matrix in $\Lambda_{2m}^m$.
Within $D$, there are $m$ $0$'s in each row from $R_2$ and $m$ $0$'s in each column from $C_2$, so $M'$ and $M$ correspond in rows from $R_2$ and columns from $C_2$. In turn, $M$ and $M'$ each have $m$ $1$'s in the first row and column within corresponding cells.
Finally, $M$ and $M'$ have $m$ $0$'s in corresponding cells of columns from $C_3$, so $M'=M$.  
Thus $D$ is a defining set for $M$. 

Next, recursively remove elements of $D$ from $M_{3,3}$ to obtain a set $C_1$ such that:
\begin{itemize} 
\item $C_1$ is a defining set of $M$; and
\item there exists no element $(r,c,e)$ of $C_1$ in $M_{3,3}$ such that $C_1\setminus \{(r,c,e)\}$ is a defining set of $M$.  
\end{itemize}
We claim that $C_1$ is in fact a critical set of $M$. It remains to show that the removal of any element 
of $C_1$ not in $M_{3,3}$ results in more than one completion. 

First consider $(1,1,1)\in C_1$. Now, $(1,m+1,0)\in M\setminus C_1$ and there exists a row $i\geq m$ such that  
$(r,m+1,1)\in M\setminus C_1$ (otherwise there are more than $m$ $0$'s in column $m+1$). 
Next observe that $(r,1,0)\in M\setminus C_1$. Thus 
$$\{(1,1,1),(1,m+1,0),(r,m+1,1),(r,1,0)\}$$
is a cycle in $M$ intersecting 
$C_1$ only at $(1,1,1)$. Swapping the entries in the cycle gives an alternate completion of $M\setminus \{(1,1,1)\}$. 

Next, let $(i,j,0)\in C_1\cap M_{2,3}$. Then there exists row $i'\in R_3$ such that $(i',j,1)\in M\setminus C_1$ (otherwise, as above, there are more than $m$ $0$'s in column $j$). Then:   
$$\{(i,j,0),(i',j,1),(i',1,0),(i,1,1)\}$$
is a cycle intersecting $C_1$ only at $(i,j,0)$. 
The case when $(i,j,0)\in C_1\cap M_{3,2}$ is similar. Finally, let $(i,j,0)\in C_1\cap M_{2,2}$. 
Let $(i',j')$ be any cell of $M_{3,3}$ containing $1$.  
Then:    
$$\{(i,j,0),(i',j',1),(i',1,0),(1,j',0),(i,1,1),(j,1,1)\}$$
intersects $C_1$ only at $(i,j,0)$. 
\end{proof}

\begin{lemma}
Let $M$ be any element of $\Lambda_{2m}^m$. 
There exists a critical set  $C_2$ of $M$ including:
\begin{itemize}
\item no elements from $M_{1,1}\cup M_{1,2}\cup M_{1,3}\cup M_{2,1}\cup M_{3,1}$;
\item each $1$ from $M_{3,3}\cup M_{2,3}\cup M_{3,2}$;
\item no $0$'s from $M_{2,2}\cup M_{2,3}\cup M_{3,2}\cup M_{3,3}$. 
\end{itemize}
\label{ian2}
\end{lemma}

\begin{proof}
First, let $D$ be the subset of $M$ defined as above with the extra property that {\em every} $1$ from 
$M_{3,3}$ is included. 

We claim that $D$ is a defining set for $M$. 
Let $M'$ be any completion of $D$ to a matrix in $\Lambda_{2m}^m$.
Within $D$, there are $m$ $1$'s in each row from $R_3$ and $m$ $1$'s in each column from $C_3$, so $M'$ and $M$ correpsond in rows from $R_3$ and columns from $C_3$. In turn, $M$ and $M'$ each have $m$ $0$'s in the first row and column within corresponding cells.
Finally, $M$ and $M'$ have $m$ $1$'s in corresponding cells of columns from $C_2$, so $M'=M$.  
Thus our claim is true. 
 
Recursively remove elements of $D$ from $M_{2,2}$ to obtain a set $C_2$ such that:
\begin{itemize} 
\item $C_2$ is a defining set of $M$; and
\item there exists no element $(r,c,e)$ of $C_2$ in $M_{2,2}$ such that $C_2\setminus \{(r,c,e)\}$ is a defining set of $M$.  
\end{itemize}
We claim that $C_2$ is in fact a critical set of $M$. It remains to show that the removal of any element 
of $C_2$ not in $M_{2,2}$ results in more than one completion. 

First, let $(i,j,1)\in C_2\cap M_{3,3}$. Then, the intersection of  
 $$\{(i,j,1),(i,1,0),(1,j,0),(1,1,1)\}$$
with  $C_2$ is $\{(i,j,1)\}$. 
Next, let $(i,j,1)\in C_2\cap M_{2,3}$. Let $(i,j')$ be a cell of $M_{2,2}$ containing $0$. 
Then  the intersection of  
 $$\{(i,j,1),(i,j',0),(1,j,0),(1,j',1)\}$$
with  $C_2$ is $\{(i,j,1)\}$. 
The case $(i,j,1)\in C_2\cap M_{3,2}$ is similar. 
\end{proof}

Now we are ready to prove Theorem \ref{maineee}. Let $\alpha$ be the number of $0$'s in subarray $M_{2,2}$. Then there
are $m(m-1)-\alpha$ $0$'s in $M_{2,3}$ and in turn, $\alpha$ $0$'s in $M_{3,3}$. Thus there are $m^2-\alpha$ $1$'s in $M_{3,3}$.  
Let $C_1$ and $C_2$ be critical sets given by Lemmas \ref{ian1} and \ref{ian2} above. 
Then $C_1$ and $C_2$ are disjoint and:
$$|C_1|+|C_2|\geq |M_{2,3}|+|M_{3,2}|+\alpha+(m^2-\alpha)+1= 3m^2-2m+1.$$  
Theorem \ref{maineee} follows.

\section{An upper bound for the supremum}

Define $B_{2m}\in \Lambda_{2m}^m$ to be the matrix
such that cell $(i,j)$ contains the element $(\lfloor (i-1)/m\rfloor, \lfloor (j-1)/m \rfloor)$, where $1\leq i,j,\leq 2m$. 
Observe that each quadrant of $B_{2m}$ contains either only $0$ or only $1$. 
Via Theorem \ref{handier} we can classify all critical sets in $B_{2m}$. 
This in turn will yield an upper bound for {\rm sup}$(\Lambda_{2m}^m)$. 

Let $C$ be a critical set of $B_{2m}$. Let $W$ be a walk as in Theorem \ref{handier}, with the ``corners'' of $W$ and blocks defined as in Section 2. 
Within $B_{2m}$, whenever the entries of cells $(i,j)$, $(i',j)$, $(i',j')$ are known, the entry of cell $(i,j')$ is uniquely determined. 
It follows that  
each cell of $L_{i,j}$ contains $1$ if $i+j$ is even and $0$ otherwise. 

For each $k$, $1\leq k\leq L$, let $s_k:=r_k-r_{k-1}$ and 
for each $k$, $1\leq k\leq L'$, let  $t_k:=c_k-c_{k-1}$. 
Observe that for each $i$ and $j$ such that $1\leq i\leq L$ and $1\leq j\leq L'$, 
$|L_{i,j}|=s_it_j$.  
The size of critical set $C$ is equal to:
$$
\sum_{i+j\equiv 1\mod{2}, i>j} |L_{i,j}|+
 \sum_{i+j\equiv 0\mod{2}, i<j} |L_{i,j}|.$$
The following lemma is immediate. 

\begin{lemma}
The size of any critical set $C$ of $B_{2m}$ is equal to:
$$
\sum_{i+j\equiv 1\mod{2}, i>j} s_{i}t_{j}+
 \sum_{i+j\equiv 0\mod{2}, i<j} s_{i}t_{j}.
$$
\label{sizeformula}
\end{lemma}

In the extreme case when $s_1=s_2=\dots =s_{2m}=t_1=t_2=\dots =t_{2m}=1$, 
we have the following. 

\begin{corollary}
There exists a critical set in $B_{2m}$ of size $2m^2-m$. 
\label{crrr}
\end{corollary}

As per the above corollary, we exhibit a critical set of size $28$ in $B_{8}$ in Figure \ref{suprri}. 

\begin{figure}
$$
\begin{array}{|c!{\vrule width1.6pt}c
!{\vrule width1.6pt}c!{\vrule width1.6pt}c!{\vrule width1.6pt}c!{\vrule width1.6pt}c!{\vrule width1.6pt}c!
{\vrule width1.6pt}c!{\vrule width1.6pt}}
\Xcline{1-1}{1.6pt}\Xcline{2-8}{0.8pt}
1 & \multicolumn{1}{c|}{0} &\multicolumn{1}{|c|}{\it 1} & \multicolumn{1}{|c|}{0} & \multicolumn{1}{|c|}{\it 1} & \multicolumn{1}{|c|}{0} & \multicolumn{1}{|c|}{\it 1} & \multicolumn{1}{|c|}{0} \\
\Xcline{1-1}{0.8pt}\Xcline{2-2}{1.6pt}\Xcline{3-8}{0.8pt}
\multicolumn{1}{|c|}{\it 0} &1 & \multicolumn{1}{c|}{0} & \multicolumn{1}{|c|}{\it 1} & 
\multicolumn{1}{|c|}{0} & \multicolumn{1}{|c|}{\it 1} & 
\multicolumn{1}{|c|}{0} & \multicolumn{1}{|c|}{\it 1}  \\
\Xcline{1-2}{0.8pt}\Xcline{3-3}{1.6pt}\Xcline{4-8}{0.8pt}
\multicolumn{1}{|c|}{1} &
\multicolumn{1}{|c|}{\it 0} &
1 & \multicolumn{1}{c|}{0} & \multicolumn{1}{|c|}{\it 1} & 
\multicolumn{1}{|c|}{0} & \multicolumn{1}{|c|}{\it 1} & 
\multicolumn{1}{|c|}{0}   \\
\Xcline{1-3}{0.8pt}\Xcline{4-4}{1.6pt}\Xcline{5-8}{0.8pt}
\multicolumn{1}{|c|}{\it 0} &
\multicolumn{1}{|c|}{1} &
\multicolumn{1}{|c|}{\it 0} & 1 & \multicolumn{1}{c|}{0} & \multicolumn{1}{|c|}{\it 1} & 
\multicolumn{1}{|c|}{0} & \multicolumn{1}{|c|}{\it 1}   \\
\Xcline{1-4}{0.8pt}\Xcline{5-5}{1.6pt}\Xcline{6-8}{0.8pt}
\multicolumn{1}{|c|}{1} &
\multicolumn{1}{|c|}{\it 0} &
\multicolumn{1}{|c|}{1} &
\multicolumn{1}{|c|}{\it 0} & 1 & \multicolumn{1}{c|}{0} & \multicolumn{1}{|c|}{\it 1} & 
\multicolumn{1}{|c|}{0}    \\
\Xcline{1-5}{0.8pt}\Xcline{6-6}{1.6pt}\Xcline{7-8}{0.8pt}
\multicolumn{1}{|c|}{\it 0} &
\multicolumn{1}{|c|}{1} &
\multicolumn{1}{|c|}{\it 0} &
\multicolumn{1}{|c|}{1} &
\multicolumn{1}{|c|}{\it 0} & 1 & \multicolumn{1}{c|}{0} & \multicolumn{1}{|c|}{\it 1}   \\
\Xcline{1-6}{0.8pt}\Xcline{7-7}{1.6pt}\Xcline{8-8}{0.8pt}
\multicolumn{1}{|c|}{1} &
\multicolumn{1}{|c|}{\it 0} &
\multicolumn{1}{|c|}{1} &
\multicolumn{1}{|c|}{\it 0} &
\multicolumn{1}{|c|}{1} &
\multicolumn{1}{|c|}{\it 0} & 1 & \multicolumn{1}{c|}{0}  \\
\Xcline{1-7}{0.8pt}\Xcline{8-8}{1.6pt}
\multicolumn{1}{|c|}{\it 0} &
\multicolumn{1}{|c|}{1} &
\multicolumn{1}{|c|}{\it 0} &
\multicolumn{1}{|c|}{1} &
\multicolumn{1}{|c|}{\it 0} &
\multicolumn{1}{|c|}{1} &
\multicolumn{1}{|c|}{\it 0} & 1  \\
\hline
\end{array}$$
\label{suprri}
\caption{A critical set of size $28$ in $B_8$ (entries shown in italics)}
\end{figure}

Our next aim is to prove that there is no critical set in $B_{2m}$ of larger size. 

Since there are $m$ $1$'s and $m$ $0$'s in each row and column, we have the following: 
\begin{eqnarray}
m-\sum_{i=0}^{\lfloor (L-1)/2\rfloor}  s_{2i+1} & = & 0 \label{const1}\\   
m-\sum_{i=0}^{\lfloor (L'-1)/2\rfloor}  t_{2i+1} & = & 0 \label{const2}\\   
m-\sum_{i=1}^{\lfloor L/2\rfloor}  s_{2i} & = & 0 \label{const3} \\   
m-\sum_{i=1}^{\lfloor L'/2\rfloor}  t_{2i} & = & 0. \label{const4}
\end{eqnarray}

\begin{theorem}
The largest critical set in $B_{2m}$ has size $2m^2-m$. 
\label{upsup}
\end{theorem}

\begin{proof}
We know a critical set of such size exists by Corollary \ref{crrr}. 

We first consider a special case: $L$ is odd and $L'=L+1=2m$. Then we must have each $s_i$ and $t_j$ equal to $1$ except $s_{2k}=2$ for presciely one value of $k$. 
From Lemma \ref{sizeformula}, it then follows that $|C|=2m^2-m$.

In all other cases, we apply the method of Lagrange multipliers to maximize the formula given by 
Lemma \ref{sizeformula} subject to the constraints \ref{const1}, \ref{const2}, \ref{const3} and \ref{const4} given above.  
(Note that we don't assume the constraints: $s_i\geq 1$ and $t_j\geq 1$, $1\leq i\leq L$, $1\leq j\leq L'$.)   

In the case $L=L'$, this yields $s_1=s_2=\dots =s_{L-1}=s$ for some $s$
and  $t_2=t_3=\dots =t_{L}=t$  and $t_1=2m-(L-1)t$ for some $t$.
If $L$ is even, from the constraints it follows that $s_L=s$, $t_1=t$ and $s=t=2m/L$. 
From Lemma \ref{sizeformula}, 
we then have
$$|C|\leq (2m/L)^2(L-1)L/2=2m^2(L-1)/L\leq m(2m-1)$$
since $L \leq 2m$. 
If $L$ is odd, from the constraints it follows that $s_L=0$, $t_1=0$, so this reduces to the previous case. 

In the case $L$ is odd, $L'=L+1$ and $L+1\leq 2m-2$, by allowing $s_{L+1}$ to be defined, 
the upper bound in the previous paragraph can be applied, so that 
$$|C|\leq 2m^2L/(L+1)\leq m^2(2m-3)/(m-1)<2m^2-m.$$

Otherwise $L'=L+1$ and $L$ is even. 
Again, we apply the Lagrangian method to obtain $s_1=s_2=\dots =s_L=s$ for some $s$ and 
$t_2=t_3=\dots = t_L$ for some $t$. 
From the constraints, $t_1+t_{L+1}=t$. 
Thus $|C|\leq stL(L-2)/2+ (t_1+t_{L+1})m=stL(L-2)/2+tm.$ 
But $s,t\leq 2m/L$ so 
$$|C|\leq 2m^2(L-1)/L \leq m(2m-1).$$
\end{proof}

  \let\oldthebibliography=\thebibliography
  \let\endoldthebibliography=\endthebibliography
  \renewenvironment{thebibliography}[1]{%
    \begin{oldthebibliography}{#1}%
      \setlength{\parskip}{0.4ex plus 0.1ex minus 0.1ex}%
      \setlength{\itemsep}{0.4ex plus 0.1ex minus 0.1ex}%
  }%
  {%
    \end{oldthebibliography}%
  }


\end{document}